\documentclass{amsart}

\usepackage[normalem]{ulem}

\usepackage{float}

\usepackage{lineno}

\usepackage{graphicx}

\usepackage{soul}

\usepackage{amssymb, amsmath, amsthm, hyperref, euscript, enumerate, color}
\usepackage[dvipsnames]{xcolor}

\usepackage{amscd}

\usepackage[shortlabels]{enumitem}

\usepackage{bbm}

\usepackage{subcaption}

\usepackage{mathtools}
\usepackage{tikz-cd}

\usepackage{fullpage}
\usepackage{microtype}

\usepackage[english]{babel}
\numberwithin{equation}{section}

\theoremstyle{plain}

\newtheorem{theorem}{Theorem}[section]

\newtheorem{corollary}[theorem]{Corollary}

\newtheorem{proposition}[theorem]{Proposition}

\newtheorem{remark}[theorem]{Remark}

\newtheorem{lemma}[theorem]{Lemma}
\newtheorem{question}[theorem]{Question}

\theoremstyle{definition}
\newtheorem{definition}[theorem]{Definition}

\begin{document}

\author{Younes Benyahia}

\title{Exotically knotted 2-spheres and the fundamental groups of their complements}

\email{youneselmaamoun.benyahia@gmail.com}


\begin{abstract}
    We show that for any finitely presented group $G$, there is a simply connected closed 4-manifold containing an infinite family of topologically isotopic but smoothly inequivalent 2-links whose 2-link group is $G$. We also show that, if $G$ satisfies the necessary topological conditions, these 2-links have nullhomotopic components.
\end{abstract}

\maketitle

\section{Introduction}

Unless stated otherwise, in what follows, all manifolds are smooth, oriented and connected and their submanifolds are smooth. We denote the tubular neighborhood of a submanifold $N$ by $\nu(N)$. We begin by defining the main objects under study.

\begin{definition}\label{links and exotic links definition}
Let $M$ be a closed 4-manifold.

    \begin{itemize}
    
        \item We say that two smooth submanifolds of $M$, $N_1$ and $N_2$, are smoothly (resp. topologically) equivalent if there exists a diffeomorphism (resp. homeomorphism) $F:M\longrightarrow M$ satisfying $F(N_1)=N_2$. Moreover, If $F$ is isotopic through diffeomorphisms (resp. homeomorphisms) to the identity, we say that $N_1$ and $N_2$ are smoothly (resp. topologically) isotopic.
        \item An $n$-component 2-link $\Gamma$ in $M$ is a smooth 2-dimensional submanifold diffeomorphic to $(S^2)_1 \sqcup (S^2)_2 \sqcup ... \sqcup (S^2)_n$, the disjoint union of $n$ copies of the 2-sphere. Its 2-link group is $\pi_1 \left( M\setminus \mathring{\nu} \left( \Gamma \right) \right)$. 
        We call a 1-component 2-link a 2-knot.
        \item Let $\Gamma_1 \subset M$ be a 2-link. We say that $\Gamma_1$ is exotic if there exists a set
        \begin{equation*}\label{exotic set definition}
            \mathcal{S}:= \{ \Gamma_k \subset M \vert k\in \mathbb{N} \}
        \end{equation*}
        of 2-links that are pairwise topologically isotopic but not smoothly equivalent.
        \item A nullhomotopic 2-link is a 2-link whose components are nullhomotopic.

    \end{itemize}

\end{definition}

Next, we recall the notions of group weight and rank. 
\begin{definition}
 
     Let $G$ be a finitely generated group. Its weight $w(G)$ is the minimal cardinality of a subset of $G$ whose normal closure is $G$. If $G$ is abelian and $G_T$ is its torsion subgroup then its rank, $rank(G)$, is the torsion-free rank of the quotient $G/G_T$ (the torsion-free part). 
\end{definition}

There has been plenty of interest in exotic 2-knots and 2-links in 4-manifolds, for example \cite{Akbulut2014IsotopingSpheres, auckly2023smoothly,aucklykimmelvrub2015stable,auckly2019isotopy,bais2023recipe, hayden2020exotically, hayden2021brunnian, hayden2023one, jinmiyazawa2023gauge, torres2020exoticknots, torres2023topologicallymanygroups}.
This note investigates the possible 2-link groups for exotic 2-links in closed simply-connected 4-manifolds, cf., \cite{auckly2023smoothly}, \cite{bais2023recipe}, \cite[Theorem B and Theorem 5.2]{hayden2020exotically}, \cite{kim2006modifying}, \cite[Theorem 5.1 and Theorem 5.2]{kim2008smooth}, \cite{kimruber2008topological}, and \cite{torres2023topologicallymanygroups} for similar results in this direction. We show that any finitely presented group arises as an exotic 2-link group. We also show that if the group arises as a nullhomotopic 2-link group, it arises as a nullhomotopic exotic 2-link group for some ambient 4-manifold. Our 2-links are produced by surgery on exotic 4-manifolds, similarly to \cite{bais2023recipe}, \cite{fintushel1994fake}, \cite{torres2020exoticknots} and \cite{torres2023topologicallymanygroups}. The exotic 4-manifolds in question are constructed using results of \cite{gompf1995new} and \cite{park2007geography}.

Before we state the main result, we point out that not every finitely presented group can be a 2-link group of a nullhomotopic 2-link in a closed simply-connected manifold. In the following proposition, we state a necessary condition. Theorem \ref{main theorem} implies that it is sufficient for the group to be realized even as an exotic nullhomotopic 2-link group.

\begin{proposition}\label{obstruction proposition}
For $n\in \mathbb{N}$, let $\Gamma$ be an $n$-component nullhomotopic 2-link in a closed simply-connected 4-manifold $M$. Let $G$ be its 2-link group, then $G$ is non-trivial and
    \begin{equation}\label{obstruction}
        w(G)=rank(H_1 ( G ) )=n .
    \end{equation}
    \end{proposition} 

    The main result of this note is the following theorem. 

\begin{theorem}\label{main theorem}
    For any finitely presented group $G$, there is a closed simply-connected 4-manifold $M_G$ and a $w(G)$-component exotic 2-link $\Gamma \subset M_G$ whose 2-link group is $G$. If, in addition, $G$ satisfies condition \eqref{obstruction} then $\Gamma$ is a nullhomotopic 2-link.
\end{theorem}
    
    We note that we are considering the empty 2-link to be nullhomotopic and to have trivial 2-link group. We also remark that, for arbitrary $G$, by considering the manifold $M_G \# r S^2 \times S^2$ instead of $M_G$ and the 2-link $\Gamma \sqcup r (S^2\times \{ pt\} )$ instead of $\Gamma$, we can increase the number of components of the 2-link in Theorem \ref{main theorem} to $w(G)+r$ without changing the 2-link group $G$. However, these new 2-links are not nullhomotopic.

    We do not aim to address minimizing the topology of the ambient 4-manifold. In particular, the construction involves a building block from \cite[Theorem 4.1]{gompf1995new}. As noted in \cite[Example 2]{baldridge2007symplectic}, such manifolds can have large Euler characteristic, and thus, so do the ambient manifolds $M_G$ of our construction. For example \cite{Akbulut2014IsotopingSpheres, auckly2023smoothly, aucklykimmelvrub2015stable, bais2023recipe, fintushel1994fake, hayden2020exotically, hayden2021brunnian, kim2006modifying, kim2008smooth, kimruber2008topological, torres2023topologicallymanygroups, torres2020exoticknots}
    produce exotic 2-knots and 2-links for many fundamental groups in smaller 4-manifolds. 
\begin{remark}\label{auckly-torres remark}
    We note that the constructed 2-links have trivial normal bundles. In the case of non-trivial normal bundles, it is not always known which groups can arise as 2-knot groups even if we drop the exoticness conclusion, cf. \cite[Corollary 5.8]{KRONHEIMER1993773} and \cite{hughesruberman2024simple}. We also mention that our ambient manifolds are of the form $Z\# S^2 \times S^2$ where $Z$ is a simply connected 4-manifold. As a consquence, they have trivial Seiberg-Witten invariants.
\end{remark}

\begin{remark}
    Take $G$ to be $F_n$ the free group of order $n$. The nullhomotopic exotic 2-link obtained in Theorem \ref{main theorem} is composed of topologically unknotted spheres by a result of Sunukjian \cite[Theorem 7.2]{sunukjian2015surfaces}. In \cite[Theorem A]{bais2023recipe} a 2-link with the same properties is obtained and shown to be non-trivially constructed by means of Brunnian exoticness \cite[Definition 2]{bais2023recipe}.
\end{remark}

The following paragraph provides context for seeking nullhomotopic exotic 2-links as opposed to the simpler task of seeking exotic 2-links of arbitrary homology (see Remark \ref{primitve 2-links} for how to find them).
It is unknown whether or not there exist exotic 2-knots (or any orientable surfaces) in $S^4$. But the question is settled for a setting that mimics that situation, that of nullhomologous exotic 2-knots (or tori) in bigger simply-connected 4-manifolds \cite{fintushel1994fake} \cite{hoffman2020null}, \cite{torres2020exoticknots}. In these settings, the examples address the case of the 2-knot group $\mathbb{Z}$ (the unknot group) which is topologically the most rigid one (see \cite[Theorem 11.7A]{freedman2014topology} and \cite{conwaypowellunknotting}). It is reasonable to guess that the same holds for the other (less rigid) 2-knot groups in $S^4$. This guess is implied by Theorem \ref{main theorem}: if $Z$ is a simply-connected 4-manifold, the family
$$\mathcal{G}_Z:= \{ \pi_1(Z \setminus \mathring{\nu} (K)) \vert K\subset Z \text{ a 2-knot} \}$$
satisfies \eqref{obstruction}, and therefore, we deduce the following corollary.

\begin{corollary}
    For any $G\in \mathcal{G}_Z$, there is a closed simply-connected 4-manifold $M_G$ and an exotic nullhomotopic 2-knot $K\subset M_G$ such that $\pi_1(M_G \setminus \mathring{\nu} ( K)) \cong G$.
\end{corollary}

A natural question follows:

\begin{question}\label{question1}
    Is there a 4-manifold $Z$ with 
$$\{ \pi_1(Z\setminus \mathring{\nu} (K)) \vert K\subset Z \text{ a 2-knot} \} = \{ \pi_1(Z\setminus \mathring{\nu} (K)) \vert K\subset Z \text{ an exotic 2-knot} \}.$$
\end{question}

Our construction does not produce such 4-manifold since there are 2-knot groups with arbitrarily large ranks. This affects the Euler characteristic of the ambient manifolds. 
\begin{remark}\label{auckly remark}
    Dave Auckly pointed to us that this question is partially answered by a theorem of his \cite[Theorem 1.5]{auckly2023smoothly}: for any smooth 4-manifold $M$ there exists a non-negative integer $n$ such that any homologically essential 2-knot $K$ in $M\# n(S^2\times S^2)$ is "exotic" in a related sense (cf. \cite[Definition 1.2]{auckly2023smoothly}). This is an affirmative answer for Question \ref{question1} when considering 2-knot groups of homologically essential 2-knots. Such groups are described in \cite[Section 3]{kim2008smooth}.
\end{remark}

\textbf{Acknowledgements:} I would like to thank Rafael Torres and Oliviero Malech for going through an earlier draft of this paper and providing many helpful comments. I would also like to thank them and Valentina Bais for useful conversations that ultimately motivated this paper. I thank Rafael Torres for pointing out Remark \ref{primitve 2-links}, Dave Auckly for pointing out Remark \ref{auckly remark} and both of them for pointing out Remark \ref{auckly-torres remark}. I thank Selman Akbulut for helpful comments. I also thank Allison N. Miller, Sümeyra Sakallı and the anonymous referees for providing many useful comments. This work was supported by the "National Group for Algebraic and Geometric Structures, and their Applications" (GNSAGA - INdAM).

\section{Background on some topological operations}
In this section, we describe surgery on loops and spheres and some of its topological effects. In particular, we prove Proposition \ref{obstruction proposition}. Lemma \ref{surgery group} and Lemma \ref{nullhomologous spheres} will be useful again in Section \ref{last section}.
We begin with some helpful terminology.
\begin{definition}
Let $M$ be a closed 4-manifold.
\begin{itemize}

        \item An $n$-component 1-link is a union of $n$ disjoint smooth simple loops in $M$.
        \item Let $p\in M$ and $\mathcal{L}=l_1 \sqcup l_2 \sqcup \cdots \sqcup l_n$ be a 1-link in $M$. We say that $\mathcal{L}$ is a generator 1-link if $\pi_1(M)$ is normally generated by the family of homotopy classes $\{[l_1'] , [l_2'] ,\cdots , [l_n'] \}$ where $l_i'= m^{-1} l_i m$ and $m$ is a path from $p$ to $l_i$ (up to conjugation, the homotopy class $[l_i']$ is independent of $m$).
\end{itemize}
        
Since $M$ is oriented, $\nu(l_i) \approx S^1 \times D^3$ for each component $l_i$ of $\mathcal{L}$. Therefore, it admits two possible framings.

\begin{itemize}
        \item Surgery on $M$ along a loop $l \subset M$ is the operation of removing a tubular neighborhood $\mathring{\nu}(l) \approx S^1 \times D^3$ and gluing back a $D^2 \times S^2$ to get a manifold $M^*$ i.e.
        $$M^*_{\varphi} := (M\setminus \mathring{\nu}(l) ) \cup_{\varphi} (D^2 \times S^2)$$ 
        where $\varphi: \partial (M\setminus \mathring{\nu}(l) ) \longrightarrow \partial (D^2 \times S^2)  $ is one of the two possible gluing maps (up to isotopy) determined by the framing of $l$. We will omit $\varphi$ from the notation $M_{\varphi}^*$ since, in general, we specify the framing.
        \item The sphere $ \{0\}\times S^2  \subset  D^2 \times S^2 \subset M^*$ is called the belt sphere of the surgery on $l$.
        
        \item Let $\mathcal{L}=l_1 \sqcup l_2 \sqcup ... \sqcup l_n$ be a 1-link in $M$. Define surgery on $M$ along $\mathcal{L}$ to be $n$ loop surgeries on $M$ along $l_1, l_2, ..., l_n$. The $n$ belt spheres of these surgeries are pairwise disjoint. We call their union the belt 2-link of the surgery on $M$ along $\mathcal{L}$.
    \end{itemize}
\end{definition}
The following summarizes the needed effect on the fundamental group.
\begin{lemma}\label{surgery group}
    Let $M$ be a closed 4-manifold and let $\mathcal{L}\subset M$ be a generator 1-link. Let $M^*$ be the result of surgery along $\mathcal{L}$ and $\Gamma$ its belt 2-link. Then $M^*$ is simply-connected and the 2-link group of $\Gamma$ is $\pi_1(M)$.
\end{lemma}
\begin{proof}
Let $p\in M$. Assume that $\mathcal{L}$ is composed of loops $l_1,l_2,...,l_n$.  Let $p$ be a point in $M \setminus \mathcal{L}$. Define $M_{i+1}$ to be the result of surgery on $M_{i}$ along the loop $l_{i+1}$, for $0 \leq i \leq n$, with $M_0 = M$. In particular, $M_n = M^*$. Before performing each loop surgery, we can isotope $l_i$ so that $p$ lies on $\partial \nu(l_i)$, we choose the base point of the fundamental groups to always be $p$.

    Applying the Seifert-van Kampen theorem
    $$\pi_1 (M_{i+1}) \cong \frac{\pi_1(M_i \setminus \mathring{\nu}(l_{i+1}) )* \pi_1(D^2\times S^2) )}{\langle [l_{i+1}] \rangle _N} \cong \frac{\pi_1(M_i)}{\langle  [l_{i+1}] \rangle_N}.$$

    We get that 
    $$\pi_1(M^*)\cong\frac{\pi_1(M)}{\langle  [l_{1}],[l_{2}],...,[l_{n}] \rangle_N } \cong \{ 1 \}. $$
    To determine the 2-link group of $\Gamma$ notice that by definition of surgery, $$M^*\setminus \mathring{\nu} (\Gamma) = M \setminus \mathring{\nu} (\mathcal{L}).$$ This implies that $$\pi_1 (M^*\setminus \mathring{\nu}(\Gamma)) = \pi_1(M \setminus  \mathring{\nu} (\mathcal{L})) \cong \pi_1(M) .$$
\end{proof}

If one drops the exoticness from the conclusion of Theorem \ref{main theorem}, the following proposition follows.
\begin{proposition}
    For any finitely presented group $G$, there is a closed simply-connected 4-manifold $M^*$ and a 2-link $\Gamma \subset M^*$ whose 2-link group is $G$. 
\end{proposition}

\begin{proof}
    It is a well known fact that every finitely presented group $G$ is the fundamental group of some closed 4-manifold $M^G$ (\cite[Theorem 1.2.33]{gompf19994} for example). Let $\mathcal{L}$ be a generator 1-link in $M^G$, let $M^*$ be the result of surgery on $M^G$ along $\mathcal{L}$ and let $\Gamma \subset M^*$ be the belt 2-link of this surgery. By Lemma \ref{surgery group}, the 2-link group of $\Gamma$ is $\pi_1(M^G)\cong G$.
\end{proof}
The strategy to prove Theorem \ref{main theorem} will be similar at core. However, instead of one manifold $M_G$ we will need a suitable infinite family of pairwise exotic 4-manifolds that stabilizes after a 1-link surgery.

The following lemma shows when does a 1-link surgery produce a nullhomotopic belt 2-link.

\begin{lemma}\label{nullhomologous spheres}
    Let $\mathcal{L}=l_1 \sqcup l_2 \sqcup ... \sqcup l_n$ be an $n$-component generator 1-link in a closed 4-manifold $M$. Let $[l_i] \in H_1(M)$ the homology class of $l_i$. Then surgery on $M$ along $\mathcal{L}$ yields a nullhomotopic belt 2-link if and only if $$\left( \langle [l_1],[l_2],...,[l_n] \rangle = H_1(M) \right) \cong \mathbb{Z}^n .$$ 
\end{lemma}

\begin{proof}
Let $M^*$ be the result of surgery along $\mathcal{L}$ and $\Gamma\subset M^*$ the belt 2-link. There is a family $S$ of surfaces in $(M \setminus \mathcal{L})=(M^*\setminus \Gamma)$ which generates the intersection lattice $(H_2(M),Q_M)$. Let $L$ be the unimodular sublattice of $(H_2(M^*),Q_{M^*})$ that $S \subset M^*$ generates; i.e., $L=(\langle S \rangle, Q_M)$.

Loop surgery increases the Euler characteristic by 2. Therefore, we get
    \begin{align}\label{b2 formula}
        b_2(M^*) = b_2(M) +2n-2b_1(M) 
    \end{align}

    First, assume that 
    $$\langle [l_1] , [l_2] ,...,[l_n] \rangle \cong \mathbb{Z}^n \cong H_1(M)$$
    Since $\Gamma$ is away from $S$, it is enough to show that the elements of $S$ span $H_2(M^*)$ to show that the components of $\Gamma$ are nullhomologous. Their nullhomotopy, in the simply-connected manifold $M^*$, will follow by Hurewicz theorem.
    In this case, \eqref{b2 formula} implies that $b_2(M^*)=b_2(M)$. Therefore, the unimodular sublattice $L$ has the same rank as $(H_2(M^*),Q_{M^*})$. We conclude that $L=(H_2(M^*),Q_{M^*})$; i.e., $S$ generates $H_2(M^*)$.

    For the inverse direction, assume that $\Gamma \subset M^*$ is nullhomotopic. Let $S^*$ be a family of surfaces in $M^*$ that generates the intersection lattice $(H_2(M^*),Q_{M^*})$. Each surface in $S^*$ intersects $\Gamma$ algebraically zero times since $\Gamma$ is nullhomologous. By tubing oppositely signed geometric intersections (see for example \cite[Subsection 15.2.1]{behrens2021disc}), we can assume that $S^*$ is in $M^*\setminus \Gamma=M\setminus \mathcal{L}$. The family $S^*$ generates a rank $b_2(M^*)$ unimodular sublattice of $(H_2(M),Q_M)$. However, \eqref{b2 formula} implies that $b_2(M^*) \geq b_2(M)$. We conclude that $b_2(M)=b_2(M^*)$. This implies that $b_1(M)=n$. Since $H_1(M)$ is generated by $n$ elements, it must be $\mathbb{Z}^n$.

\end{proof}

\begin{proof}[Proof of Proposition \ref{obstruction proposition}]
Notice that $\Gamma$ is the belt 2-link of surgery on a 4-manifold $\Tilde{M}$ with $\pi_1(\Tilde{M})=G$ along an $n$-component 1-link $\mathcal{L} \subset \Tilde{M}$ ($\Tilde{M}$ is the result of surgery along $\Gamma$ and $\mathcal{L}$ its belt 1-link). Since, $\mathcal{L}$ is an $n$-component generator 1-link, we also know that $n \geq w(G)$. But $w(G)\geq rank(H_1(G))$. Applying Lemma \ref{nullhomologous spheres}, we get that $rank(H_1(G))=n$. Therefore, $w(G)=rank(H_1(G))=n.$ Notice that this formula implies that, if $G$ is trivial, the 2-link $\Gamma$ has no components. Therefore, a non-empty nullhomotopic 2-link cannot have a trivial 2-link group.
\end{proof}

We describe one other useful operation (cf. \cite[Section 3.3]{gompf19994}).
\begin{definition}
    Let $M$ be a closed 4-manifold and $T\subset M$ a smoothly embedded torus with zero self intersection. Let $\sigma:S^1 \times S^1 \times D^2 \longrightarrow \nu(T) \subset M$ a framing for $T$. A 0-log transform of $M$ along the framed torus $T$ is the manifold
    $$\Tilde{M}=(M\setminus \mathring{\nu}(T)) \cup_{\varphi} S^1 \times S^1 \times D^2$$
    Where the gluing map $$\varphi: \partial (M\setminus \mathring{\nu} (T) ) = \partial \nu(T) \longrightarrow \partial(S^1 \times S^1 \times D^2)=T^2 \times \partial(D^2)$$
    is chosen such that
    $$(\sigma^{-1} \circ \varphi^{-1})_* : H_1(T^3)  \longrightarrow H_1(T^3)  $$ is a map corresponding to the matrix
$    \begin{pmatrix}
1 & 0 & 0\\
0 & 0 & 1 \\
0 & -1 & 0
\end{pmatrix} $ after identifying $H_1(T^3)$ with $\mathbb{Z}^3$.

\end{definition}

\section{A symplectic manifold }\label{gompf section}
In \cite[Theorem 4.1]{gompf1995new}, the author constructs a closed symplectic 4-manifold $X^G$ with fundamental group $G$ for any finitely presented group $G$. In this section we give a brief description of that construction.
Let $G$ be a finitely presented group with the following presentation
\begin{equation}\label{original presentation G}
     G= \langle \alpha_1 , \alpha_2, ...,\alpha_g \vert w_1, w_2,...,w_r \rangle.
\end{equation}

Consider the manifold $T^2\times \Sigma_{g+k}$ where $T^2$ is the torus and $\Sigma_{g+k}$ is a genus $g+k$ closed surface for a certain suitable $k$ (see \cite{gompf1995new} for more details). Its fundamental group has the following presentation
$$\pi_1(T^2 \times \Sigma_{g+k})\cong \langle x \rangle \oplus \langle y \rangle \oplus \langle \alpha_1 , \beta_1, \alpha_2,\beta_2, ...,\alpha_{g+k}, \beta_{g+k} \vert \prod_{i=1}^{g+k} [\alpha_i,\beta_i]  \rangle$$
where $x$ and $y$ are represented by the loops $X\times \{ p\}$ and $Y \times \{p \}$ on $T^2\times \{p\}$ and $\alpha_i$ and $\beta_i$ are represented by the loops (here, the loops are extended by a path to the base point) $\{q\} \times a_i$ and $\{q\} \times b_i$ on $\{q\} \times \Sigma_{g+k}$ for $i=1,...,g$, $p\in \left( \Sigma_{g+k} \setminus \left( \bigcup_i (a_i \cup b_i) \right) \right)$ and $q\in (T^2 \setminus (X\cup Y))$, as in Figure \ref{figure sigma_g}. The goal is to perform cut and paste operations on $T^2 \times \Sigma_{g+k} $ in order to obtain a symplectic manifold with fundamental group $G$.
\begin{figure}[h]
    \centering
\includegraphics[width=0.85\linewidth]{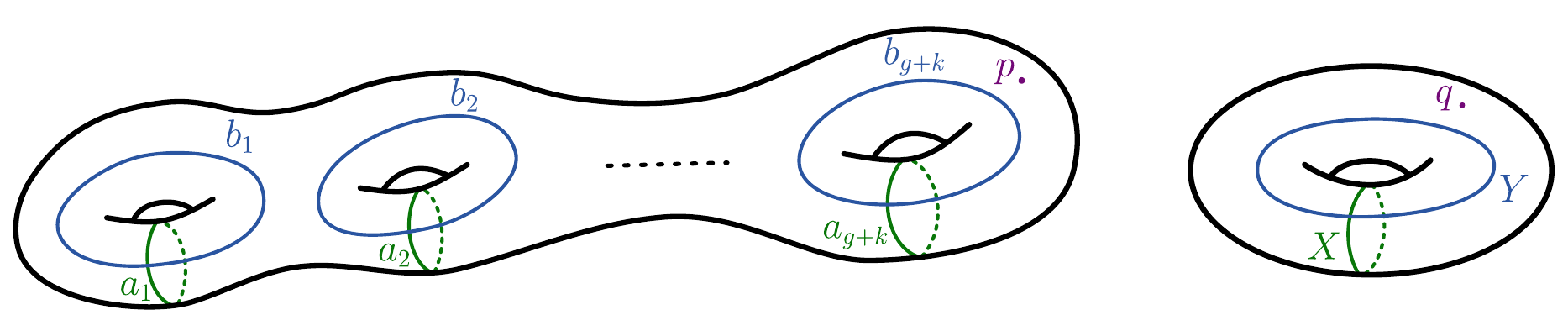}
    \caption{$\Sigma_{g+k}$ and $T^2$}
    \label{figure sigma_g}
\end{figure}

We know that the $K3$ surface contains two disjoint Lagrangian tori $T$ and $T'$ with $$\pi_1 ( K3 \setminus \mathring{\nu}(T \sqcup T') )\cong \{1\}$$ (see, for example, \cite[Section 3.1]{gompf19994}). The two tori are independent in homology. Therefore, \cite[Lemma 1.6]{gompf1995new} implies that, after a perturbation of the symplectic form, both tori are symplectic.
In order to eliminate the extra generators and the relators in the fundamental group, it is possible to symplectically sum multiple copies of the elliptic surface $K3$ along $T$ to $T^2\times \Sigma_{g+k}$  each copy along tori $X\times Y$,  $X\times b_i$  for $i=1,...,g+k$, $X\times a_j$ for $j=g+1,...,g+k$ and $X\times W_l$ for $l=1,...,r$ where $W_l$ are loops representing the words $w_l$ written with the elements $\{ \alpha_1, \alpha_2,...,\alpha_g \} \subset \pi_1(T^2\times \Sigma_{g+k})$. Denote the resulting manifold by $X^G$. Using the Seifert-van Kampen theorem, we get that
$$\pi_1(X^G)\cong \frac{\pi_1(T^2\times \Sigma_{g+k})}{\langle x,y,\beta_1,...,\beta_{g+k}, \alpha_{g+1},...,\alpha_{g+k},w_1,...,w_r\rangle_N} \cong G.$$
\begin{remark}
    
Notice that any copy of $K3$ contains a copy of $T'$ disjoint from $T$. Therefore, $X^G$ contains a copy of $T'$.
\end{remark}

In the rest of the paper, we will use a specific presentation of $G$ to construct $X^G$:
Let $m:=w(G)$ and suppose that $\{ x_1,...,x_m \}$ is a normally generating set for $G$. Notice that
    \begin{align}\label{new presentation}
        G\cong\langle \alpha_1 , \alpha_2,...,\alpha_g , \alpha_{g+1} ,...,\alpha_{g+m} \vert w_1,...,w_r,x_1^{-1}\alpha_{g+1} , ...,x_m^{-1} \alpha_{g+m}\rangle
    \end{align}
This presentation has the nicer property that it makes $G$ normally generated by $m$ of its generators: 
\begin{align*}
    \frac{G}{\langle \alpha_{g+1},...,\alpha_{g+m} \rangle_N} &\cong \frac{\langle \alpha_1 , \alpha_2,...,\alpha_{g+m} \vert w_1,...,w_r,x_1^{-1}\alpha_{g+1} , ...,x_m^{-1} \alpha_{g+m}\rangle}{\langle \alpha_{g+1},...,\alpha_{g+m} \rangle_N} \\
    &\cong \frac{\langle \alpha_1,...,\alpha_g \vert w_1,...,w_r \rangle}{\langle x_1,...,x_m \rangle_N} \\
    &\cong \{1\}.
\end{align*}

We point out a lemma which will be useful in computing fundamental groups after a 0-log transform.
\begin{lemma}\label{pi_1 of 0-log transform}
     Let $T_i := X\times a_i \subset X^G$ be tori equipped with their Lagrangian framing,
     for $i=g+1,...,g+m$. Denote the (based) loop $\{q\} \times a_i \subset X^G \setminus \mathring{\nu}\left( \bigsqcup_{i=g+1}^{g+m} T_i \right) $ by $A_i$ for $i=g+1,...,g+m$. We have that
    $$ \frac{\pi_1 \left( X^G \setminus \mathring{\nu} \left( \bigsqcup_{i=g+1}^{g+m} T_i \right) \right)}{\langle [A_{g+1}] ,...,[A_{g+m}] \rangle_N} \cong \{1 \}$$

\end{lemma}
\begin{proof}
    Since we are using the group presentation \eqref{new presentation}. We know that $$\frac{\pi_1(X^G )}{\langle  [A_{g+1}] ,...,[A_{g+m}] \rangle_N} \cong \frac{G}{\langle \alpha_{g+1} ,...,\alpha_{g+m} \rangle_N} \cong \{1 \}.$$
    By the Seifert-van Kampen theorem, we also know that
    $$\pi_1(X^G) \cong \frac{\pi_1 \left( X^G \setminus \mathring{\nu}\left( \bigsqcup_{i=g+1}^{g+m} T_i \right)\right)}{\langle [\mu_{g+1}] ,...,[\mu_{g+m}] \rangle_N} $$
     where $\mu_i$ is the (based) meridian of the framed torus $T_i$.
    The two formulas imply that
    \begin{align*}
        \{1 \} &\cong \dfrac{\pi_1(X^G )}{\langle[A_{g+1}] ,...,[A_{g+m}] \rangle_N} \\
        &\cong \dfrac{\pi_1 \left( X^G \setminus \mathring{\nu}\left(\bigsqcup_{i=g+1}^{g+m} T_i \right) \right)}{\langle [A_{g+1}] ,...,[A_{g+m}],[\mu_{g+1}] ,...,[\mu_{g+m}] \rangle_N}.
    \end{align*}

    Therefore, to prove the lemma, it is enough to show that 
    $$\textstyle{[\mu_i] =1 \in \pi_1 \left(X^G \setminus \mathring{\nu}\left(\bigsqcup_{i=g+1}^{g+m} T_i \right) \right) \text{, for $i=g+1,...,g+m$}}$$

    \begin{figure}[h]
        \centering
        \includegraphics[scale=0.5]{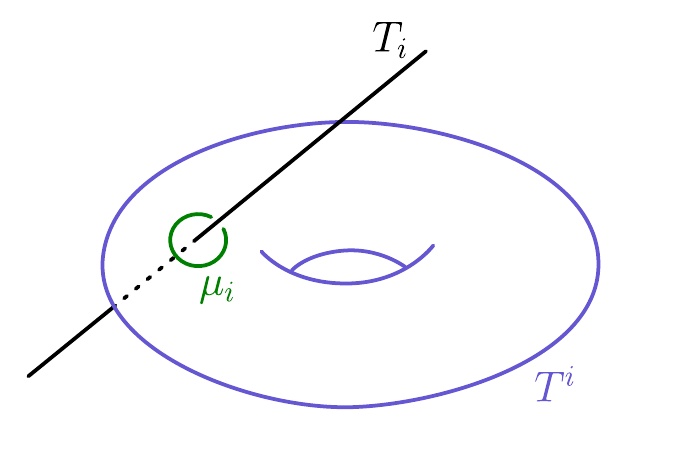}
        \caption{The meridian on the dual torus}
        \label{dual tori}
    \end{figure}
    If $b_i'$ is a parallel loop of $b_i$ in $\Sigma_{g+m+k}$ then the torus $T^i := Y\times b_i' \subset X^G$ is geometrically dual to $T_i$  and disjoint from $T_j$ for $i\neq j \in \{ g+1,...,g+m \}$. This implies that in the complement $X^G \setminus \mathring{\nu}\left(\bigsqcup_{i=g+1}^{g+m} T_i\right) $ the loop $\mu_i$ is homotopic to a (based) loop representing a commutator of the form $[Y^{\pm 1},b_i'^{\pm 1}]$ (a small loop on the punctured dual torus, see Figure \ref{dual tori}). However, $[Y^{\pm 1},b_i'^{\pm 1}]=[cy^{\pm 1}c^{-1},d\beta_i^{\pm 1} d^{-1}]$ with $c$ and $d$ potentially arising due to the basing of the loops in the fundamental group and $y$ and $\beta_i$ the loops discussed at the beginning of this section. But the fiber sums with $K3$ makes $y$ and $\beta_i$ trivial; therefore,
    $$[\mu_i]= [cc^{-1},d d^{-1}] = 1$$
    Which concludes the proof of the lemma.

\end{proof}

\section{Exotic manifolds}\label{park section}
In this section, for any finitely presented group $G$, we will describe an infinite family $$\{ M_k^G \vert k\in \mathbb{N} \} $$ of closed pairwise exotic 4-manifolds, i.e., pairwise non diffeomorphic but homeomorphic manifolds, whose fundamental group is $G$. To construct it, we will follow the core ideas present in \cite{park2007geography}. However, we will not investigate the Euler characteristic and signature of the family, so we can reduce the construction to a simpler version. First, we restate \cite[Proposition 2.2]{park2007geography}.
    
    \begin{proposition}[\cite{park2007geography}]\label{proposition 2.2 park}
        Suppose $Z$ is a symplectic 4-manifold which contains a symplectic torus $T$ lying in a cusp neighborhood. Then, for each integer $m\geq2$, there is a family of symplectic 4-manifolds $$\{ Z\#_T E(m)_K \vert K \text{ is a fibered knot in } S^3 \}$$ which are mutually non-diffeomorphic but all homeomorphic. Furthermore, $$\pi_1(Z\#_T E(m)_K)\cong \pi_1(Z).$$
    \end{proposition}

    The family in Proposition \ref{proposition 2.2 park} is constructed, first, by considering the elliptic surface $E(m)$ for $m\geq 2$ which contains two smoothly embedded tori $\mathfrak{T}_1$ and $\mathfrak{T}_2$ such that $$\pi_1(E(m) \setminus \mathring{\nu} (\mathfrak{T}_1 \sqcup \mathfrak{T}_2)) = \{ 1 \}.$$    
    
    Along $\mathfrak{T}_1$ one performs Fintushel-Stern knot surgeries \cite{Fintushel1996KnotsLA}, using a certain infinite family of knots $\mathcal{K}$, to obtain an infinite family $\{ E(m)_K \vert  K \in \mathcal{K} \}$ of simply-connected 4-manifolds that are pairwise homeomorphic but non-diffeomorphic.
    Next, one fiber sums $E(m)_K$ and the manifold $Z$ along tori $\mathfrak{T}_2 \subset E(m)_K$ and $T \subset Z$. This yields the manifold $Z\#_T E(m)_K$.

The 4-manifold $X^G$ is symplectic and contains the symplectic near-cusp embedded torus $T'$; therefore, we can apply Proposition \ref{proposition 2.2 park} for $Z=X^G$ with $m=3$ to obtain an infinite family of closed pairwise exotic manifolds. We denote this family by 
\begin{align}\label{family of exotic manifolds}
    \{ M_k^G \vert k\in \mathbb{N} \}
\end{align}

In the next proposition we mention a couple of useful observations that follow from the proof of Proposition \ref{proposition 2.2 park} found in \cite{park2007geography}.
    \begin{proposition}[\cite{park2007geography}] \label{park's theorem}
    The manifolds in the family \eqref{family of exotic manifolds}
        satisfy the following
        \begin{enumerate}
           
            \item\label{inclusion of pi1} the inclusion $i: X^G \setminus \mathring{\nu}(T') \hookrightarrow M_k^G $ induces a surjective map $$i_* : \pi_1  (X^G \setminus \mathring{\nu}(T')) \longrightarrow \pi_1(M_k^G).$$
            \item for any positive integer $k$, there exists a homeomorphism $f_k:M_k^G \longrightarrow M_1^G$ that restricts to the identity on $X \setminus \mathring{\nu}(T')$. 
        \end{enumerate}
    \end{proposition}
    \begin{remark}\label{non-spin remark}
        We can assume that $\mathfrak{T}_1 \subset E(m)$ is geometrically dual to a sphere $S$ of self intersection $m$ which is disjoint from $\mathfrak{T}_2 \subset E(m)$. Hence, $X\#_T E(m)$ contains the sphere $S$. If $m$ is odd then $X\#_T E(m)$ is non-spin and, therefore, so is $X\#_T E(m)_K$. Since $m=3$ for the family \eqref{family of exotic manifolds}, its elements are non-spin.
    \end{remark}

\section{Proof of the main theorem}\label{last section}

In this section, we will prove Theorem \ref{main theorem} but first, we start with the following fact.

\begin{lemma}\label{stabilizator}
    Let $K$ be a knot in $S^3$ and $Y_K$ the result of a 0-Dehn surgery along $K$. Let $m \subset( S^3 \setminus \nu(K)) \subset Y_K$ be a meridian of $K$ equipped with the standard 0-framing. Perform two loop surgeries on $S^1 \times Y_K$ along the framed loops $\{p\}\times m$ and $S^1 \times \{p'\}$ to obtain a manifold $B$, where $p\in S^1$ and $p' \in Y_K$. Then $B$ is diffeomorphic to either $S^2\times S^2$ or $\mathbb{CP}^2 \# \overline{\mathbb{CP}^2}$. 
\end{lemma}

\begin{proof}
    Using the Moishezon trick \cite[Lemma 13]{moishezon1977complex}, explained in \cite[Lemma 3]{gompf1991sums} (cf. \cite[section 3.1]{bais2023recipe} or \cite{torres2020exoticknots}), we will break down the surgery along $m$ to two steps. The first step is a 0-log transform and the second is a loop surgery, Figure \ref{moishezon trick} depicts the local handlebody diagram of the operation.

\begin{figure}[h]
    \centering
    \includegraphics[scale=0.21]{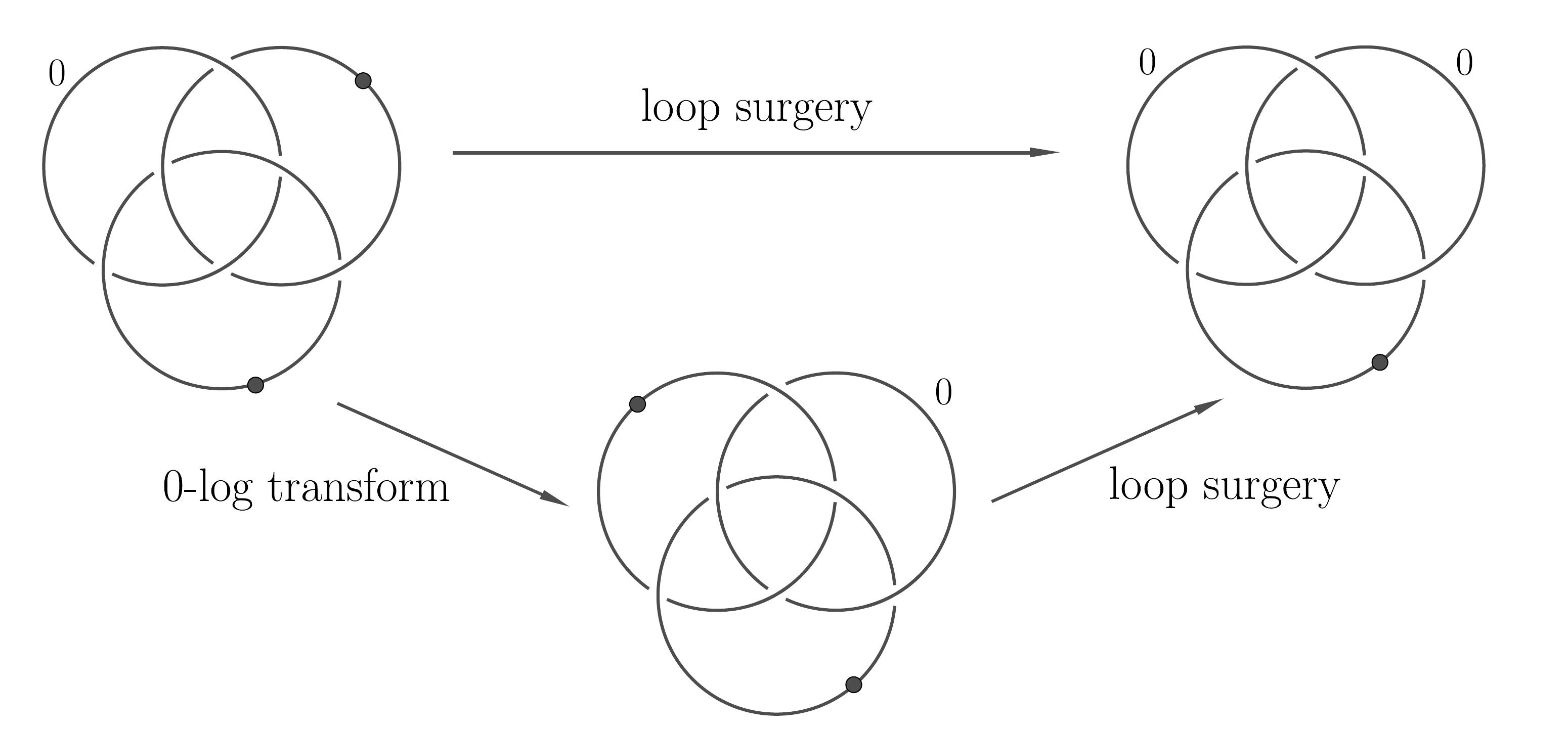}
    \caption{Splitting a loop surgery on $T^2 \times D^2$}
    \label{moishezon trick}
\end{figure}
 The 0-log transform is performed along the torus $S^1 \times m$ equipped with the 0-framing of $m$ in $ (S^3\setminus \nu(K)) \subset Y_K$. On the other hand, since $Y_K$ is the result of a 0-Dehn surgery on $S^3$, $S^1 \times Y_K$ is the result of a 0-log transform on $S^1 \times S^3$ along the torus $S^1\times K$ equipped with the 0-framing of $K$ in $ S^3$. Therefore, performing another 0-log transform on $S^1\times Y_K$ along the framed torus $S^1\times m$ is equivalent to trivially reversing the earlier torus surgery done along $S^1\times K \subset S^1\times S^3$. Hence, by the Moishezon trick, the result of loop surgery on $S^1\times Y_K$ along $m$ is diffeomorphic to the result of a certain loop surgery on $S^1 \times S^3$. However, to obtain $B$, we still need to perform a loop surgery along $S^1\times \{ p'\}$ which yields $S^4$ when performed on $S^1 \times S^3$. Thus, the result of the two loop surgeries on $S^1\times Y_K$ along $\{ p \} \times m$ and $S^1\times \{p'\}$ is the same as the result of a loop surgery on $S^4$ i.e. either $S^2\times S^2$ or $\mathbb{CP}^2 \# \overline{\mathbb{CP}^2}$. 
\end{proof}

Now, in order to prove Theorem \ref{main theorem}, let $G$ be a finitely presented group. 
We construct $X^G$ according to the presentation \eqref{new presentation}. Let 
$$\mathcal{L}=(\{q\} \times a_{g+1}) \sqcup  (\{q\} \times a_{g+2}) \sqcup ... \sqcup( \{q\} \times  a_{g+m})$$ 
be a generator 1-link in $X^G$ where $\{q\} \times a_i$ is a loop representing $\alpha_i \in \pi_1(X^G)$ as in Figure \ref{figure sigma_g}; by item \ref{inclusion of pi1} of Proposition \ref{park's theorem}, its inclusion in $M_k^G$, denoted $\mathcal{L}_k$, is also a generator 1-link in $M_k^G$. Surgery on $\mathcal{L}_k$ yields a simply-connected manifold $M_k^*$ and a belt 2-link $\Gamma_k \subset M_k^*$ whose 2-link group is $G$ (by Lemma \ref{surgery group}).

The first step of the proof is to show that this surgery stabilizes the family \eqref{family of exotic manifolds}.
\begin{lemma}\label{ambient manifold}
    The diffeomorphism type of the manifold $M_k^*$ is independent of $k$; i.e., for all $k\in\mathbb{N}$, there exists a diffeomorphism $\varphi_k:M_k^* \longrightarrow M_1^*$
\end{lemma}

\begin{proof}
    Let $X'$ be a parallel of $X$ in $T^2$ containing $q$. Since each component of the 1-link $\mathcal{L}$ lies on a torus $T_i=X' \times a_i$,  we will break down the loop surgery using the Moishezon trick (see the proof of Lemma \ref{stabilizator} and Figure \ref{moishezon trick}).   
    
    The 0-log transform is performed on the framed torus $T_i$ for $i=g+1,...,g+m$. Call the resulting manifold $\Tilde{M_k^G}$. By the Seifert-van Kampen theorem and Lemma \ref{pi_1 of 0-log transform}, we have: 
    \begin{align*}
        \pi_1(\Tilde{M_k^G}) &\cong \frac{\pi_1\left(M_k^G\setminus \mathring{\nu} \left(\bigsqcup_{i=g+1}^{g+m} T_i \right)\right)}{\langle [a_1'], ...,[a_g'] \rangle_N}  && \text{ by the Seifert-van Kampen theorem} \\
        &\cong \frac{\pi_1\left(E(3)_k \#_{T^2} \left( X^G\setminus \mathring{\nu} \left(\bigsqcup_{i=g+1}^{g+m} T_i \right)\right)\right)}{\langle [a_1'], ...,[a_g'] \rangle_N} \\
        &\cong \frac{\pi_1\left( X^G\setminus \mathring{\nu}\left(\bigsqcup_{i=g+1}^{g+m} T_i \right)\right)}{\langle [a_1'], ...,[a_g'] \rangle_N} && \text{ by the Seifert-van Kampen theorem}\\ 
        &\cong \{ 1\} && \text{ by Lemma \ref{pi_1 of 0-log transform}}
    \end{align*}
    
    The second step states that the manifold $M_k^*$ is diffeomorphic to the result of surgery on $\Tilde{M_k^G}$ along an $m$-component 1-link $\mathcal{L}^*$. However, since $\Tilde{M_k^G}$ is simply-connected and non-spin, the result of surgery on the nullhomotopic 1-link is $\Tilde{M_k^G} \# m(S^2\times S^2)$. We have 
    $$M_k^* \approx \Tilde{M_k^G} \#m(S^2\times S^2) \approx \left(E(3)_k\# m(S^2\times S^2) \right) \#_{T^2}  \Tilde{X^G} $$ 
    where $\Tilde{X^G}$ is the result of 0-log transform on $X^G$ along the tori $T_i$ for $i=g+1,...,g+m$. On the other hand, $E(3)_k$ is obtained as a fiber sum between $E(3)$ and $S^1 \times Y_K$, along the tori $\mathfrak{T}_1 \subset E(3)$ and $S^1\times m$. Cobordism arguments in \cite{mandelbaum1979decomposing}, in particular, the lemma stated in \cite[Lemma 4]{gompf1991sums} (cf. \cite{akbulut2002variations}, \cite{auckly2003families} and \cite{baykur2018dissolving} for similar stabilization results) shows that there is a diffeomorphism
    $$h:\left(\left(E(3)_k \# m(S^2 \times S^2) \right)\setminus \mathring{\nu}(\mathfrak{T}_2)  \right) \longrightarrow \left( \left( E(3)\# Y' \# (m-1)(S^2\times S^2) \right) \setminus \mathring{\nu}(\mathfrak{T}_2) \right)$$
     where $Y'$ is the result of two loop surgeries on $S^1 \times Y_K$ along $\{p\}\times m$ and $S^1 \times \{p'\}$. Since $E(3)$ contains an embedded sphere with odd self intersection (Remark \ref{non-spin remark}), we can alter the surgery framing on $\{p\}\times m$ to match the one described in Lemma \ref{stabilizator}; this is done by dragging a small segment of the loop around the sphere and bringing it back to its initial position with a different framing (see \cite[Proposition 5.2.4]{gompf19994}). So, by Lemma \ref{stabilizator}, $Y'$ is either $S^2\times S^2$ or $\mathbb{CP}^2 \# \overline{\mathbb{CP}^2}$ and since $E(3)$ and $E(3)_k$ are non-spin (Remark \ref{non-spin remark}, we can assume that $E(3) \# Y'$ is diffeomorphic to $E(3) \# (S^2 \times S^2)$. Moreover, \cite[Lemma 4]{gompf1991sums} states that $h$ restricts to the identity on the boundary $\partial \nu  (\mathfrak{T}_2)$.  This implies that the smooth equivalence extends after the gluing to yield
    $$M_k^* \approx \left(E(3)_k\# m(S^2\times S^2)\right) \#_{T^2}  \Tilde{X^G} \approx   \left(E(3)_1\# m(S^2\times S^2)\right) \#_{T^2}  \Tilde{X^G} \approx M_1^* $$
    
    Denote by $\varphi$ the diffeomorphism $M_k^* \longrightarrow M_1^*$. This concludes the proof.
\end{proof} 

The second step of the proof is to address the exoticness of the belt 2-links $\Gamma_k$
\begin{proposition}\label{exotic links construction}
    Recall that $\Gamma_k \subset M_k^*$ is the belt 2-link of the surgery on $M_k^G$ along the generator 1-link $\mathcal{L}_k \subset M_k^G$. The 2-link $\Gamma_1 \subset M_1^*$ is exotic.
\end{proposition}

\begin{proof}
    By Proposition \ref{park's theorem}, there exists a homeomorphism $$f_k : M_k^G\longrightarrow M_1^G$$
    which restricts to the identity on $X^G \setminus \mathring{\nu}(T')$. Therefore, it sends the framed neighborhood of the generator 1-links $\mathcal{L}_k \subset M_k^G$ to the framed neighborhood of $\mathcal{L}_1\subset M_1^G$ i.e. we have a homeomorphism of pairs 
    $$f_k:(M_k^G, \nu (\mathcal{L}_k )) \rightarrow (M_1^G, \nu(\mathcal{L}_1))$$ 
    which yields another homeomorphism of pairs $$\Tilde{f}_k: (M_k^*,\Gamma_k)\longrightarrow (M_1^*,\Gamma_1)$$ where $\Gamma_k$ is the belt 2-link of the surgery along $ \mathcal{L}$ in $M_k^*$. On the other hand, by Proposition \ref{ambient manifold}, there is a diffeomorphism $$\varphi_k:M_k^* \longrightarrow M_1^* .$$ 
    Therefore, $$\Tilde{f}_k  \circ \varphi_k^{-1}: (M_1^*, \varphi(\Gamma_k)) \longrightarrow (M_1^*,\Gamma_1)$$ is a homeomorphism of pairs.
    Since $M_1^*$ is of the form $N\#(S^2\times S^2)$ and has indefinite intersection form, a theorem of Wall \cite[Theorem 2]{wall1964diffeomorphisms} implies the existence a self diffeomorphism 
    $$\psi_k:M_1^* \longrightarrow M_1^*$$  
    such that the induced map on the second homology group
    $$(\Tilde{f}_k \circ \varphi_k^{-1} \circ \psi_k )_*:H_2(M_1^*) \longrightarrow H_2(M_1^*)$$
    is the identity map. Thus, results of Quinn \cite{quinn1986isotopy} and Perron \cite{perron1986pseudo} allow us to conclude that the map $$g_k:=\Tilde{f}_k \circ \varphi_k^{-1} \circ \psi_k : (M_1^* , \psi_k^{-1} \circ \varphi_k (\Gamma_k)) \longrightarrow (M_1^*, \Gamma_1)$$ is topologically isotopic to the identity. We conclude that the elements of the family 
    \begin{equation}\label{link family}
        \{ \psi_k^{-1} \circ \varphi_k (\Gamma_k) \vert k\in \mathbb{N} \}
    \end{equation}
    are all topologically isotopic. For brevity, denote $\psi_k^{-1} \circ \varphi_k (\Gamma_k) $ by $\Tilde{\Gamma}_k$ (with $\Tilde{\Gamma}_1 =\Gamma_1$).

    It remains to show that the elements of the family \eqref{link family} are not smoothly equivalent. Assume for the sake of contradiction that there exists two distinct positive integers $k$ and $j$ for which there is a diffeomorphism of pairs 
    $$\phi: (M_1^*, \Tilde{\Gamma}_k) \longrightarrow (M_1^*, \Tilde{\Gamma}_j).$$
    This implies the existence of the following diffeomorphism of pairs
    $$\varphi_j^{-1} \circ \psi_j \circ \phi \circ \psi_k^{-1} \circ \varphi_k: (M_k^*,\Gamma_k) \longrightarrow (M_j^*,\Gamma_j). $$
    Since 2-spheres in 4-manifolds admit a unique framing (up to isotopy), the result of surgery on $M_k^*$ along the 2-link $\Gamma_k$ is diffeomorphic to the result of the analogous surgery on $M_j^*$ along $\Gamma_j$. However, the surgery results are $M_k^G$ and $M_j^G$ which cannot be diffeomorphic by Proposition \ref{proposition 2.2 park}. Hence, it is impossible for $\Tilde{\Gamma}_k$ to be smoothly equivalent to $\Tilde{\Gamma}_j$.
     
\end{proof}

What is still missing to conclude Theoreom \ref{main theorem} is addressing the nullhomotopy of the 2-link.
\begin{proof}[Proof of Theorem \ref{main theorem}]
    Notice that if $G$ satisfies the condition \eqref{obstruction} then, by Lemma \ref{nullhomologous spheres}, the exotic 2-link constructed above is nullhomotopic. Theorem \ref{main theorem} follows by taking $M_G$ to be $M_1^*$ and $\Gamma$ to be $\Gamma_1$.
\end{proof}

\begin{remark}\label{primitve 2-links}
    We point out that for constructing exotic 2-links that are not necessarily nullhomotopic, Lemma \ref{ambient manifold} can be obtained in a simpler way: It is enough to add enough trivial components to $\mathcal{L} \subset M_k^G$ which, after surgery, will correspond to $S^2\times S^2$ stabilizations. Lemma \ref{ambient manifold} follows by Wall's stabilization theorem \cite{Wall1964stabilization}.
   
\end{remark}

\begin{remark}
    Finally, we note that any $n$-component 2-link $\Gamma$ in a closed simply-connected 4-manifold whose 2-link group satisfies Condition \ref{obstruction} is nullhomotopic. This can be seen by looking at $\Gamma$ as a belt 2-link and applying Lemma \ref{nullhomologous spheres}. Therefore, provided we can make 2-links with any 2-link group, the main challenge to making these 2-links nullhomotopic is in minimizing their number of components $n$ to fit Condition \ref{obstruction}. For example, this is not achieved in Remark \ref{primitve 2-links}.
\end{remark}

\bibliographystyle{myhalpha}
\bibliography{references} 

\end{document}